\documentclass[11pt, oneside]{amsart}
\usepackage[english]{babel}
\usepackage{hyperref}
\usepackage{amssymb, amsfonts, amsmath}
\usepackage{pgfplots}
\usepackage{comment} 

\usepackage{mathtools}
\usepackage{wrapfig}
\usepackage{tikz}
\usepackage{comment}
\usepackage{mathrsfs}
\usepackage{tabularx, hyperref}
\usepackage{amssymb} \usepackage{amsfonts} \usepackage{amsmath}
\usepackage{amsthm} \usepackage{epsfig, subfig}
\usepackage{ amscd, amsxtra, latexsym}
\usepackage{epsfig,  graphicx, psfrag}
\usepackage[all]{xy}
\usepackage{caption}
\usepackage{enumerate}
\usepackage{color}
\DeclareMathOperator{\im}{im}
\DeclareMathOperator{\id}{id}
\DeclareMathOperator{\mcg}{MCG}

\DeclareMathOperator{\fv}{FV}

\DeclareMathOperator{\spl}{Sp}
\newcommand{\str}{{{\mathrm{str}}}}
\newcommand{\hand}{{{\mathrm{hand}}}}
\newcommand{\spin}{{\mathrm{Sp}}}
\newcommand{\cminus} {{\setminus\hspace{-2pt}\setminus}}
\newcommand{\fvz}{{{\fv_{\Z}}}}
\newcommand{\N}{\ensuremath {\mathbb{N}}}
\newcommand{\R} {\ensuremath {\mathbb{R}}}

\newcommand{\Z} {\ensuremath {\mathbb{Z}}}
\newcommand{\fil}{\ensuremath{\textrm{fill}}}
\newcommand{\norm}[1]{{\left\|#1\right\|}}
\newcommand{\norml}[1]{{\norm{#1}_{1}}}
\newcommand{\fnorm}[1]{{\norm{#1}_{\fil,\Z}}}

\newcommand\restr[2]{{
\left.\kern-\nulldelimiterspace 
#1 
\vphantom{\big|} 
\right|_{#2} 
}}

\newcommand\rrestr[2]{{
\left.\kern-\nulldelimiterspace 
\left.\kern-\nulldelimiterspace 
#1 
\vphantom{\big|} 
\right|\hspace{-2.4pt} 
\right|_{#2} 
}}

\renewcommand{\phi}{\varphi}
\renewcommand{\colon }{\,:}

\newtheorem{lemma}{Lemma}[section]
\newtheorem{thm}[lemma]{Theorem}
\newtheorem{prop}[lemma]{Proposition}

\newtheorem*{prop*}{Proposition}
\newtheorem{prop_intro}{Proposition}

\newtheorem{quest_intro}[prop_intro]{Question}
\newtheorem{thm_intro}[prop_intro]{Theorem}

\newtheorem{cor_intro}[prop_intro]{Corollary}

\theoremstyle{definition}

\newtheorem{defn}[lemma]{Definition}

\theoremstyle{remark}
\newtheorem{rmk}[lemma]{Remark}
\newtheorem{step}{Step}

\newtheoremstyle{citing}
{3pt}
{3pt}
{\itshape}
{}
{\bfseries}
{.}
{.5em}
{\thmnote{#3}}
\theoremstyle{citing}
\newtheorem*{varthm}{}

\makeatletter
\newcommand\thankssymb[1]{\textsuperscript{\@fnsymbol{#1}}}
\makeatother

\begin{document}

\title[]{Integral filling volume, complexity and integral simplicial volume of $3$-dimensional mapping tori}


\author[Federica Bertolotti]{Federica Bertolotti\thankssymb{1}}\thanks{\thankssymb{1}Partially supported by the INdAM GNSAGA Project, CUP E55F22000270001.}
\address{Scuola Normale Superiore, Pisa, Italy}
\email{federica.bertolotti@sns.it}

\author[]{Roberto Frigerio\thankssymb{2}}\thanks{\thankssymb{2}Partially supported by INdAM through GNSAGA, and by MUR through the PRIN project
n. 2022NMPLT8 ``Geometry and topology of manifolds'', CUP I53D23002370001}
\address{Dipartimento di Matematica, Universit\`a di Pisa, Italy}
\email{roberto.frigerio@unipi.it}



\begin{abstract}
	We show that the integral filling volume of a Dehn twist $f$ on a closed oriented surface vanishes, i.~e.~that the integral simplicial volume of the mapping torus with monodromy $f^n$
	grows sublinearly with respect to $n$. We deduce a complete characterization of mapping classes on surfaces with vanishing integral filling volume and, building on results by Purcell and Lackenby on the complexity of mapping tori, we show that, in dimension three, complexity and integral simplicial volume are not Lipschitz equivalent.
\end{abstract}

\maketitle

\section{Introduction}
Unless otherwise stated, in this paper every manifold is understood to be just a topological manifold, and is assumed to be oriented, closed and connected.
The complexity of a manifold $M$ may be measured by many numerical invariants. Among them, one of the most natural is probably the \emph{$\Delta$-complexity} $\Delta(M)$, i.~e.~the minimal number of simplices in a (loose) triangulation of $M$ (see Definition~\ref{complexity:def}; if $M$ does not admit any triangulation, then we understand that
$\Delta(M)=\infty$). Any triangulation of $M$ gives rise to an integral fundamental cycle of $M$,
and an interesting invariant of $M$ is the minimal number of simplices in (or, more precisely, the $\ell^1$-norm of) an integral fundamental cycle for $M$ in singular homology: this positive integer
is the \emph{integral simplicial volume} of $M$, and it is usually denoted by $\|M\|_\Z$ (see Definition~\ref{sv:def}). Integral cycles are more flexible than triangulations, and real cycles are more flexible than integral ones. One may then take a step further and consider the infimum of the $\ell^1$-norms of \emph{real} fundamental cycles of $M$ in singular homology, thus getting the \emph{simplicial volume} $\|M\|$ of $M$ (see again Definition~\ref{sv:def}),
which first appeared in the proof of Mostow Rigidity Theorem described in~\cite{thurston_geometry_1979} and attributed by Thurston to Gromov (see also~\cite{bourbaki}), 
and was then thoroughly studied as an invariant of independent interest in Gromov's pioneering paper~\cite{Gromov82}.

While the simplicial volume is closely related to the geometric structures that a manifold can carry, $\Delta$-complexity and integral simplicial volume are more combinatorial in nature, and may be exploited to estimate algebraic topological invariants: for example,
both $\Delta$-complexity and integral simplicial volume provide upper bounds for Betti numbers, hence for the Euler characteristic. Since the Euler characteristic is multiplicative with respect to finite coverings, it is bounded from above also by the stable $\Delta$-complexity and the stable integral simplicial volume (which are obtained by stabilizing the corresponding invariants over finite coverings). These facts, together with a long-standing conjecture by Gromov on the relationship between the simplicial volume and the Euler characteristic for aspherical manifolds, have driven a lot of attention to stable integral simplicial volume, which has recently been proven to be equal to the classical simplicial volume for all aspherical $3$-manifolds~\cite{stable1} (in higher dimension, stable integral simplicial volume is not equal to simplicial volume even on the smaller class of hyperbolic manifolds~\cite{stable2}).

\subsection*{$\Delta$-complexity versus integral simplicial volume}
In this paper we concentrate our attention on the (rather unexplored) comparison between $\Delta$-complexity and integral simplicial volume.
Let us say that two numerical invariants $h_1,h_2$ on the class of closed oriented $n$-manifolds are \emph{equivalent} if there exists a constant $k\geq 1$ such that
$h_1(M)\leq k\cdot h_2(M)$ and $h_2(M)\leq k\cdot h_1(M)$ for every closed oriented $n$-manifold $M$.

While the $\Delta$-complexity and the integral simplicial of a manifold are obviously positive, the simplicial volume vanishes on very large classes of manifolds (e.g.~on manifolds with an amenable fundamental group). As a consequence, the simplicial volume is not equivalent to the integral simplicial volume, neither to the $\Delta$-complexity.
The integral simplicial volume and the $\Delta$-complexity are equivalent (in fact, they coincide!) in dimension one (trivially, since the only closed $1$-manifold is $S^1$ and
$\Delta(S^1)=\|S^1\|_\Z=1$) and in dimension two, since $\Delta(S^2)=\|S^2\|_\Z=2$ and $\Delta(\Sigma_g)=\|\Sigma_g\|_\Z=4g-2$ for every closed oriented
surface of genus $g\geq 1$ (see~\cite[Proposition 4.3]{Claraodd}). In this paper we prove the following:

\begin{thm_intro}\label{nonequivalent}
	In dimension $3$, integral simplicial volume and $\Delta$-complexity are not equivalent. More precisely, there exists a sequence $M_i$ of $3$-manifolds such that
	$$
		\lim_{i\to \infty} \frac{ \Delta(M_i)}{\|M_i\|_\Z}=+\infty \ .
	$$
\end{thm_intro}

The manifolds appearing in the sequence described in  Theorem~\ref{nonequivalent} are all torus bundles over the circle. Indeed,
if $N$ is an $n$-dimensional manifold and $f\colon N\to N$ is a homeomorphism, we denote by
$N\rtimes_f  S^1$  the $(n+1)$-dimensional manifold fibering over $S^1$ with monodromy $f$, i.~e.
\[N\rtimes_f  S^1 =\frac{N \times [0,1]} {\sim}\ ,\]
where $(x,0)\sim(f(x),1)$ for all $x \in N$.
Let now $T=S^1\times S^1$ be the 2-dimensional torus and let $f\colon T\to T$ be a Dehn twist along a homotopically non-trivial simple closed curve in $T$.
Then the manifold
$M_i$ mentioned in Theorem~\ref{nonequivalent} is the mapping torus of $f^i$, i.~e.~$M_i=T \rtimes_{f^i} S^1$.

Our proof of  Theorem~\ref{nonequivalent} relies on the (separate) study of the $\Delta$-complexity and of the integral simplicial volume of mapping tori. In fact we prove that, if $f\colon T\to T$
is a Dehn twist on the torus, then the $\Delta$-complexity of $T \rtimes_{f^i} S^1$ grows linearly with respect to $i$ (see Theorem~\ref{thm: triangulation complexity of nilmanifolds}), while
its  integral simplicial volume grows sublinearly with respect to $i$ (see Corollary~\ref{torus:cor}), and this of course suffices to deduce Theorem~\ref{nonequivalent}.

\subsection*{$\Delta$-complexity of $3$-manifolds fibering over the circle}
The $\Delta$-com\-plex\-i\-ty of manifolds fibering over the circle has been extensively studied by Lackenby and Purcell in~\cite{lackenby_triangulation_2019,lackenby_triangulation_2022}. Exploiting their techniques and building on their results we prove here the following:

\begin{thm_intro}\label{thm: triangulation complexity of nilmanifolds}
	Let $f \colon T \to T$ be a self-homeomorphism of the torus representing an infinite-order element in $\mcg(T)$. Then  the $\Delta$-complexity  of the mapping torus of $f^i$ has linear growth with respect to $i$, i.~e.~there exist constants $0<k_1\leq k_2$ such that
	$$
		k_1\leq \frac{\Delta \left(T \rtimes_{f^i} S^1\right)}{i}\leq k_2
	$$
	for every $i\geq 1$.
\end{thm_intro}
Notice that the homeomorphism $f \colon T \to T$ satisfies the hypothesis of the theorem if and only it $f$ is a non-trivial power of a Dehn twist or an Anosov homeomorphism.
In particular, when $f\colon T\to T$ is an Anosov automorphism of the torus, the conclusion of Theorem~\ref{thm: triangulation complexity of nilmanifolds} is proved to hold
in~\cite[Theorem 1.4]{lackenby_triangulation_2022}, where explicit estimates for $k_1$ and $k_2$ are also provided. Hence, our contribution mainly consists in adapting Lackenby's and Purcell's strategy to the Dehn twist case.


\subsection*{Filling volume and integral simplicial volume}
In order to study the integral simplicial volume of $T \rtimes_{f^i} S^1$ we make use of an invariant recently introduced in~\cite{bertolotti_length_2022}: the filling volume.

Let $f\colon M\to M$ be an orientation preserving self-homeomorphism of a manifold $M$.
The integral (resp.~real) filling volume $\fv_\mathbb{Z}(f)$ (resp.~$\fv_{\mathbb{R}}(f)$) of  $f$ is a numerical invariant
defined in terms of the action of $f$ on fundamental cycles of $M$ (see Definition~\ref{fv:def}).
In the real case, Theorem 5 of~\cite{bertolotti_length_2022} shows that the filling volume of a homeomorphism $f$ is equal to the simplicial
volume of the mapping torus of $f$, i.~e.~that
\begin{equation}\label{simplvol:eq}
	\fv_\R(f)=\|M\rtimes_f S^1\|\ .
\end{equation}

The proof of~\cite[Theorem 5]{bertolotti_length_2022} may be easily adapted to show the following result, which will be proved in Section~\ref{filsimpl:sec}:

\begin{thm_intro}\label{thm: IntegralFillingVolumeAndCovering}
	Let  $f \colon  M \to M$ be an orientation preserving homeomorphism of a manifold $M$.
	Then, \[\fv_\Z(f)=\lim_{i \to \infty} \frac{\|M\rtimes_{f^i} S^1\|_\Z} {i}.\]
\end{thm_intro}

Equation~\eqref{simplvol:eq} provides a satisfactory description of the real filling volume of a homeomorphism,
and has been exploited, for example, in order to classify self-homeomorphisms of surfaces with positive (resp.~vanishing) real filling volume
(see~\cite[Corollary 9]{bertolotti_length_2022}). On the contrary,
the integral filling volume seems to be a bit more elusive. In Section~\ref{surfaces:sec} we prove the following result, which completely describes the behaviour of the integral filling volume
on the mapping class group of surfaces:

\begin{thm_intro}\label{filvol:surfaces}
	Let $S$ be the closed oriented surface of genus $g$ and let
	$f\colon S\to S$ be an orientation-preserving self-homeomorphism of $S$. Then
	$\fvz(f)>0$ if and only if one of the following conditions is satisfied:
	\begin{itemize}
		\item  $g = 1$ and $f$ is Anosov;
		\item $g>1$ and there exist a natural number $k\in\mathbb{N}$ and a subsurface $S'\subseteq S$ (that is, a codimension-0  submanifold $S'$ of $S$, possibly with boundary) such that $f^k$ is isotopic to a homeomorphism
			$h\colon S \to S$ that satisfies $h(S')=S'$ and restricts to a pseudo-Anosov self-homeomorphism $h|_{S'}$ of $S'$.
	\end{itemize}
\end{thm_intro}

The key step in our approach to Theorem~\ref{filvol:surfaces} is the proof that the integral filling volume of a  Dehn twist on the torus vanishes (see Theorem~\ref{thm:fv_Z-Dehn}).
We then extend this result to Dehn twists on higher genus surfaces (Theorem~\ref{thm: DehnTwistSuperficie}) and we finally conclude
the proof of Theorem~\ref{filvol:surfaces} by exploiting some non-vanishing results
for the \emph{real} integral  filling volume from~\cite{bertolotti_length_2022} (which imply the corresponding non-vanishing results for the integral filling volume).

Theorem~\ref{filvol:surfaces} and~\cite[Corollary 9]{bertolotti_length_2022} imply that, for every self-homeomor\-phism $f$ of a surface $S$ of genus $g\geq 2$,
we have $\fvz(f)=0$ if and only if $\fv_\mathbb{R}(f)=0$.  However, it is shown in~\cite{bertolotti_length_2022} that, if $f$ is an Anosov automorphism of the torus,
then  $\fv_\mathbb{R}(f)=0$, while $\fvz(f)>0$ (and also in higher dimensions
the vanishing of the real filling volume does not ensure the vanishing of the integral filling volume, see~\cite[Corollary 13]{bertolotti_length_2022}). In particular, the integral filling volume succeeds in distinguishing an Anosov automorphism of the torus from a Dehn twist, while the real filling volume does not.

Putting together Theorems~\ref{thm: IntegralFillingVolumeAndCovering} and~\ref{filvol:surfaces} we readily deduce the following:

\begin{cor_intro}\label{torus:cor}
	Let  $f\colon T\to T$ be the Dehn twist along a homotopically non-trivial simple closed curve of the torus. Then
	$$
		\lim_{i \to \infty} \frac{\|M\rtimes_{f^i} S^1\|_\Z} {i}= \fv_\Z(f)=0\ ,
	$$
\end{cor_intro}
As explained above, Corollary~\ref{torus:cor} and Theorem~\ref{thm: triangulation complexity of nilmanifolds} finally imply Theorem~\ref{nonequivalent}.

\subsection*{A glimpse into higher dimensions}

In higher dimensions, the equivalence between  integral simplicial volume and $\Delta$-complexity fails in dramatic ways: 
for example, when an $n$-manifold $M$ does not admit a triangulation (in which case  $\Delta(M)=\infty$, while of course $\|M\|_\Z$ is finite), or
when infinitely many pairwise non-homeomorphic $n$-manifolds share the same integral simplicial volume.

In dimension $3$ it is known that both these phenomena cannot occur: every $3$-manifold admits a triangulation, and a result by Clara L\"oh ensures that, for every $k\in\mathbb{N}$,
up to homeomorphism there exist only finitely many $3$-manifolds with integral simplicial volume equal to $k$~\cite[Proposition 5.3]{Claraodd}. However, in dimension $n\geq 4$ there exist non-triangulable manifolds,
and there also exist infinite collections of pairwise non-homeomorphic $n$-manifolds sharing the same integral simplicial volume. 
Indeed, since homotopy equivalent manifolds share the same integral simplicial volume, the last sentence readily follows from 
\cite[Theorem 1.2]{infinite}, where the authors construct, for every $n\geq 4$, an infinite collection
of closed, orientable,
topological $n$-manifolds that are all simple homotopy equivalent to each other but pairwise not
homeomorphic (if $n>4$, these manifolds can be chosen to be smooth, hence triangulable). 

To the best of the authors' knowledge, the following question is still open:

\begin{quest_intro}\label{finite-to-one:quest}
	Let $n\geq 4$ and $k\in\mathbb{N}$ be fixed. Is the number of homotopy classes of $n$-manifolds $M$ such that
	$$
		\|M\|_\mathbb{Z}\leq k
	$$
	finite?
\end{quest_intro}

Recall that a manifold $M$ \emph{$d$-dominates} a manifold $N$ if there exists a degree-$d$ map $f\colon M\to N$.
As observed by L\"oh, in dimension 3 the fact that the integral simplicial volume is finite-to-one follows from the 
fact that, for any fixed $3$-manifold and any fixed $d>0$, 
the number of the homeomorphism classes of $3$-manifolds that are $d$-dominated by $M$ is finite~\cite{Liu}. Also in higher dimensions Question~\ref{finite-to-one:quest} is related
to the (probably difficult) question whether a fixed manifold $M$ could $d$-dominate an infinite number of pairwise non-homotopic manifolds, or not.

\subsection*{Plan of the paper} In Section~\ref{filsimpl:sec} we recall the definitions of filling volume and of simplicial volume, and we prove Theorem~\ref{thm: IntegralFillingVolumeAndCovering}. In Section~\ref{surfaces:sec} we compute the integral filling volume of the Dehn twist on the torus, and we then prove Theorem~\ref{filvol:surfaces}.
Finally,
Section~\ref{complexity:sec} is devoted to the proof of Theorem~\ref{thm: triangulation complexity of nilmanifolds} concerning
the $\Delta$-complexity of $3$-dimensional nilmanifolds and solmanifolds.

\section{Filling volumes and simplicial volumes}\label{filsimpl:sec}
Let $R=\R$ or $\Z$, and
let $M$ be an $n$-dimensional manifold (recall that, in this paper, every manifold is assumed to be closed, connected and oriented). Let $C_*(M,R)$ denote the complex of singular chains on $M$ with coefficients in $R$, and
for every $i\in\mathbb{N}$ denote by $Z_i(M,R)\subseteq C_i(M,R)$ (resp.~$B_i(M,R)\subseteq C_i(M,R)$) the subspace of degree-$i$ cycles (resp.~boundaries).
We endow $C_*(M,R)$  with the usual $\ell^1$-norm such that, if $c=\sum_{k\in I} a_k\sigma_k$ is a singular chain written in reduced form, then
$$
	\|c\|_1=\left\| \sum_{k\in I} a_k\sigma_k\right\|=\sum_{k\in I} |a_k|\ .
$$
On the space $B_i(M,R)$ of boundaries there is also defined the \emph{filling norm} $\|\cdot \|_{\fil, R}$ such that, for every $z\in B_i(M,R)$,
$$
	\|z\|_{\fil,R} =\inf \{\|b\|_1\, |\, b\in C_{i+1}(M,R)\, ,\ \partial b=z\}\ .
$$
Recall that $H_n(M,\Z)\cong \Z$ is generated by the \emph{fundamental class} $[M]_\Z$ of $M$, and
that an integral fundamental cycle for $M$ (or $\Z$-fundamental cycle) is just any representative of $[M]_\Z$. We denote by $[M]_\R\in H_n(M,\R)$ the \emph{real} fundamental class of $M$, i.~e.~the image of the fundamental class $[M]_\Z$
via the change of coefficient map $H_n(M,\Z)\to H_n(M,\R)$. An $\R$-fundamental cycle (or \emph{real} fundamental cycle) of $M$ is any representative of $[M]_\R$ in $Z_n(M,\R)$.

\begin{defn}\label{sv:def}
	The \emph{$R$-simplicial volume} of $M$ is
	$$
		\|M\|_R=\inf \{\|z\|_1\, |\, z\ \text{is an}\ R-\text{fundamental cycle for}\ M\}\ .
	$$
\end{defn}

The real simplicial volume $\|M\|_\R$ is just the classical simplicial volume, and it is denoted simply by $\|M\|$.

If $f\colon M\to M$ is a map, we denote by $f_*\colon C_*(M,R)\to C_*(M,R)$ the induced map on singular chains.
Observe that $f_*$
is norm non-increasing both with respect to the $\ell^1$-norm,
and (on the subspace of boundaries) with respect to the filling norm.

The following invariant, called the \emph{filling volume} of $f$, was defined in~\cite{bertolotti_length_2022}.

\begin{defn}\label{fv:def}
	Let $f\colon M\to M$ be an orientation preserving self-home\-o\-mor\-phism
	and let $z\in C_n(M,R)$ be an $R$-fundamental cycle for $M$.
	We set
	$$
		\fv_R(f)=\lim_{m\to \infty} \frac{ \|f^m_*(z)-z\|_{\fil,R}}{m}\ ,
	$$
	where $\|\cdot \|_{\fil,R}$ denotes the filling norm on $B_n(M,R)$. (The symbol $\fv$ stands for \emph{filling volume}).

It is proved in~\cite{bertolotti_length_2022} that
the invariant $\fv_R(f)$ is well defined, i.~e.~the above limit (which exists and is finite thanks to Fekete's Lemma) does not depend on the choice of the fundamental cycle $z$.
Moreover,
it turns out that $\fv_R(f)$ only depends on the homotopy class of $f$ (see  again \cite{bertolotti_length_2022}), and
it follows from the very definition that $\fv_\R(f)\leq \fv_\Z(f)$.
\end{defn}

\begin{rmk}\label{FVchar}
	Fekete's Lemma and the fact that  $\fv_R(f)$ does not depend on the choice of the fundamental cycle $z$
	also imply the following characterization of the filling volume, which will be exploited in the proof of
	Theorem~\ref{thm: IntegralFillingVolumeAndCovering} (see~\cite[Proposition 2]{bertolotti_length_2022} for the details):
		$$
	\fv_R(f)=\inf\left\{ \frac{ \|f^m_*(z)-z\|_{\fil,R}}{m}\ \big|\ z\ \text{fundamental cycle for}\ M,\, m\in\mathbb{Z}\right\}\ .
	$$
\end{rmk}

Recall that the symbol $M\rtimes_{f} S^1$ denotes the mapping torus of $f$,  i.~e.~the manifold obtained from
$M\times [0,1]$ by identifying $(x,0)$ with $(f(x),1)$ for every $x\in M$.
The following very neat relationship between $\fv_\mathbb{R} (f)$ and the mapping torus of $f$ is proved in~\cite{bertolotti_length_2022}:

$$
	\fv_\R (f)=\|M\rtimes_{f} S^1\|\ .
$$

By adapting the proof of~\cite[Theorem 5]{bertolotti_length_2022}, we prove here an analogous (but slightly more involved) relationship between the \emph{integral} filling volume of $f$ and the \emph{integral} simplicial
volume of its mapping torus:

\begin{varthm}[Theorem~\ref{thm: IntegralFillingVolumeAndCovering}]
	Let  $f \colon  M \to M$ be an orientation preserving homeomorphism of a manifold $M$.
	Then \[\fv_\Z(f)=\lim_{i \to \infty} \frac{\|M\rtimes_{f^i} S^1\|_\Z} {i}.\]
\end{varthm}
\begin{proof}
	Let $z$ be an integral fundamental cycle for $M$. It is not difficult to construct a (relative) fundamental cycle for $M\times [0,1]$
	of norm $(n+1)\|z\|_1$ whose boundary is equal to $(i_0)_*(z)-(i_1)_*(z)$, where $i_j\colon M\to M\times \{j\}$ is the inclusion. In order to
	``close up'' this fundamental cycle to obtain a fundamental cycle for $M\rtimes_{f} S^1$ one may sum it with $(i_1)_*(b)$,
	where $b\in C_{n+1}(M,\mathbb{Z})$ is an element such that $\partial b=z-f_*(z)$. This gives the estimate
	$$
		\|M\rtimes_{f} S^1\|_\mathbb{Z}\leq  (n+1)\|z\|_1+ \|f_*(z)-z\|_{\fil,\mathbb{Z}}
	$$
	(see \cite[Lemma 3.1]{bertolotti_length_2022} for a detailed proof).
	By replacing the homeomorphism $f$ with a power $f^i$ and dividing by $i$ we obtain
		$$
		\frac{\|M\rtimes_{f^i} S^1\|_\mathbb{Z}}{i}\leq \frac{(n+1)\|z\|_1}{i}+\frac{\|f^i_*(z)-z\|_{\fil,\Z}}{i}\ .
	$$
	By taking the limit in this inequality as $i\to +\infty$, we then get the inequality
	\[\limsup_{i \to \infty} \frac{\|M\rtimes_{f^i} S^1\|_\Z} {i} \leq \fv_\Z(f)\ .\]

	In order to conclude, we are now left to show that
	$$ \fv_{\mathbb{Z}} (f)\leq
		\liminf_{i \to \infty} \frac{\|M\rtimes_{f^i} S^1\|_{\mathbb{Z}}}{i} \ .
	$$
	Of course, it is sufficient to show that
	$$
		\fv_{\mathbb{Z}} (f)\leq \frac{\|M\rtimes_{f^i} S^1\|_{\mathbb{Z}}}{i}
	$$
	for every $i\in\mathbb{N}$.

	To this aim,
	we rely on the proof of~\cite[Theorem 5]{bertolotti_length_2022}, which shows that, if
	$z\in C_{n +1}(M\rtimes_f S^1,\mathbb{Z})$ is an integral fundamental cycle such that $\|z\|_1=\|M\rtimes_f S^1\|_\mathbb{Z}$,
	then one can find $m\in\mathbb{N}$ big enough and
	a chain $w\in C_{n+1}(M,\mathbb{Z})$ such that:
	\begin{enumerate}
		\item $ \|w \|_1 \leq m \cdot \|z \|_1$, and
		\item $\partial w = b - f^{m}_* (b)$ for some fundamental cycle $b$ of $M$.
	\end{enumerate}
	(While~\cite[Theorem 5]{bertolotti_length_2022} concerns the case with real coefficients, the proof of this specific fact works with any coefficients.) As a consequence, we
	get
	$$
		\fv_\mathbb{Z}(f)\leq \frac{\|b - f^{m}_* (b)\|_{\fil,\mathbb{Z}}}{m}\leq \frac{\|w\|_1}{m}\leq \|z\|_1=\|M\rtimes_f S^1\|_{\mathbb{Z}}\ ,
	$$
	where the first inequality is due to the characterization of the filling volume described in Remark~\ref{FVchar}.
	For every $i\in\mathbb{N}$ we may replace $f$ with $f^i$ in the inequality above, thus getting
	$$
		\fv_\mathbb{Z}(f^i)\leq\|M\rtimes_{f^i} S^1\|_{\mathbb{Z}}\ .
	$$
	Finally, by exploiting the fact that $\fv_{\mathbb{Z}}(f^i)=i\cdot \fv_{\mathbb{Z}}(f)$ (see~\cite[Theorem 2]{bertolotti_length_2022}), we may deduce that
	$$
		\fv_{\mathbb{Z}} (f)\leq \frac{\|M\rtimes_{f^i} S^1\|_{\mathbb{Z}}}{i}\ ,
	$$
	as desired. This concludes the proof.

\end{proof}
\begin{rmk}
The proof of Theorem \ref{thm: IntegralFillingVolumeAndCovering} does not exploit any peculiar property of integral coefficients. Indeed, 
if $R$ is a normed ring with unity, then the notions of fundamental cycle, $\ell^1$-norm and $R$-filling norm make sense. Hence, one can define the
$R$-simplicial volume $\|\cdot \|_R$ and the filling volume $FV_R(\cdot)$, and the statement of 
Theorem \ref{thm: IntegralFillingVolumeAndCovering} holds true with $\mathbb{Z}$ replaced by $R$. (Since the classical simplicial
volume is multiplicative with respect to finite coverings, when $R=\mathbb{R}$  Theorem~\ref{thm: IntegralFillingVolumeAndCovering}
just gives back~\cite[Theorem 5]{bertolotti_length_2022}, which states
that $FV_{\mathbb{R}}(f)= \|M\rtimes_f S^1\|$.)
\end{rmk}

\section{The integral filling volume of surface automorphisms}\label{surfaces:sec}

This section is devoted to the study of the integral filling volumes of self-homeomorphisms of surfaces.
We first show that the integral filling volume vanishes for Dehn twists on a torus (Theorem~\ref{thm:fv_Z-Dehn}).
Then, we use this result to prove
that the same holds for Dehn twists on any closed orientable surface (Theorem~\ref{thm: DehnTwistSuperficie}). Finally, we put together these results with the description of the \emph{real} filling volume of surface automorphisms given in~\cite{bertolotti_length_2022} to obtain a complete characterization of the surface automorphisms with  vanishing
integral filling volume (Theorem~\ref{filvol:surfaces}).

\subsection{Filling Volume of Dehn twists on the torus}

We devote the whole subsection to the proof of the following result:

\begin{thm}\label{thm:fv_Z-Dehn}
	Let $T$ be the torus and let $f\colon  T \to T$ be a Dehn twist along any homotopically non-trivial simple closed curve of $T$. Then $\fv_\Z(f)= 0$.
\end{thm}

In order to prove the theorem,
we fix a fundamental cycle $c \in C_2(T,\Z)$ of the torus, and we construct efficient chains that fill the boundaries $c - f_*^{2^{2k}} (c)$.

Let $e_i \in \R^{k +1}$ be the $i$-th vector in the canonical basis of $\R^{k+1}$ and $\Delta^k \subset \R^{k+1}$ the convex hull of $e_1,\ldots,e_{k +1}$, that is
\[\Delta^k =\left\{\left.\sum_{i = 1}^{k +1}\lambda_ie_i \right|0 \leq \lambda_i \leq 1, \ \sum_{i = 1}^{k +1}\lambda_i = 1 \right\}\ .\]

For $n,m \in \N$ and $p_1,\ldots,p_{n +1} \in \R^m$, we denote by $\str_m(p_1,\ldots,p_{n +1})$ the straight $n$-simplex in $\R^m$ with vertices $p_1,\ldots,p_{n +1}$, that is
the singular simplex defined by the map
\[\begin{matrix}
		\str_m(p_1,\ldots,p_{n +1})\colon & \Delta^n                        & \to     & \R^m                               \\
		                                  & \sum_{i = 1}^{n+1} \lambda_ie_i & \mapsto & \sum_{i = 1}^{n+1} \lambda_ip_i\ .
	\end{matrix}\]

Let $T$ be the torus $\R^2/\Z^2$, $\pi \colon  \R^2 \to T$ the quotient map. We denote by
\[[p_1,\ldots,p_{n +1}]\in C_n(T,\Z)\]
the $n$-simplex in $T$ given by $\pi \left(\str_2(p_1,\ldots,p_{n +1})\right)$.

\begin{rmk}\label{rmk: translation invariant}
	As $T =\R^2/\Z^2$, we have that for every $i,j \in \Z$ the two symbols $[(x_1,y_1),(x_2,y_2),\ldots,(x_{n +1},y_{n +1})]$ and $[(x_1 +i,y_1 + j),(x_2 + i,y_2 + j),\ldots,(x_{n +1} + i,y_{n +1} + j)]$ represent the same simplex in $C_n(T,\Z)$.
\end{rmk}
With this notation, a fundamental cycle of the torus is given by
\[c = [(0,0),(1,0),(1,1)] - [(0,0),(0,1),(1,1)] \]
(see Figure \ref{fig: c0}).

Let $f \colon  T=\R^2/\Z^2 \to T =\R^2/\Z^2$ be a Dehn twist along a homotopically non-trivial simple closed curve. Up to conjugation by an element of the mapping class group, we can suppose $f$ is represented by the matrix
\[\begin{pmatrix} 1&1 \\ 0&1 \end{pmatrix}\ .\]
Let $f_*\colon C_*(T,\Z)\to C_*(T,\Z)$ be the induced map on the set of chains.

We have
\[f_*^{k}(c)= [(0,0),(1,0),(k +1,1)]- [(0,0),(k,1),(k +1,1)],\]
that is a fundamental cycle of $T$ as well (Figure \ref{fig: c0}).

\begin{figure}[h]
	\includegraphics{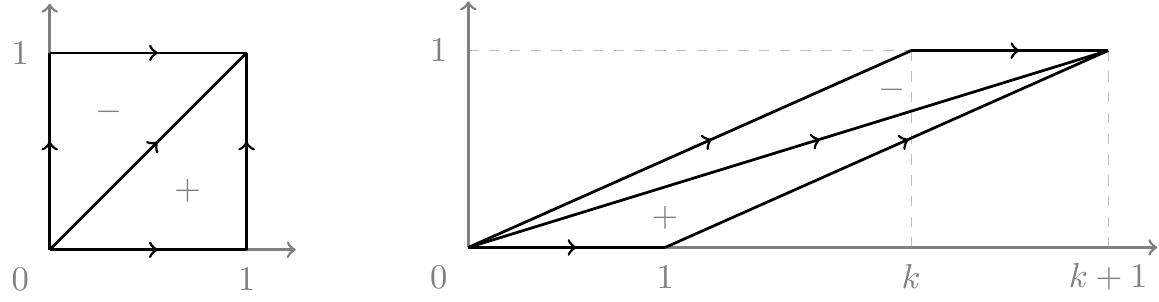}
	\caption{
		A lift of the cycles $c$ and $f_*^k(c)$ in $\R^2$.
	}\label{fig: c0}
\end{figure}

The chain $f_*^{2^{2k}} (c) - c \in C_2(T,\Z)$ is a boundary as it is the difference of two fundamental cycles.
For $k \in \N$, we will construct chains ${w_{k}} \in C_3(T,\Z)$ such that $\partial w_k=f_*^{2^{2k}} (c) - c$ and with the property that  $\|w_{k}\|_1$ grows linearly with $k$.
This will prove that the filling norm of $f_*^{{k}} (c)- c$ grows less than linearly.

We subdivide the construction in steps.
\begin{step}\label{step: collasso}
	Let
	\begin{align*}
		\tau_k = & + [(0,0),(1,0),(2^{2k}+1,0),(2^{2k}+1,1)]      \\
		         & - [(0,0),(2^{2k},0),(2^{2k}+1,0),(2^{2k}+1,1)] \\
		         & + [(0,0),(2^{2k},0),(2^{2k},1),(2^{2k}+1,1)]
	\end{align*}
	(see Figure~\ref{prysm}) and set
	\[b_k =[(0,0), (1,0),(2^{2k}+1,0)]- [(0,0),(2^{2k},0),(2^{2k} +1,0)]\ .\]
	Then we have $\norml{\tau_k}= 3$ and  $\partial \tau_k + b_k =f_*^{2^{2k}} (c) - c$, since
	\begin{align*}
		\partial \tau_k + b_k = & 
+ [(1,0),(2^{2k}+1,0),(2^{2k}+1,1)] 
- [(0,0),(2^{2k},0),(2^{2k},1)]\\
								&+ [(0,0),(1,0),(2^{2k}+1,1)]
								- [(0,0),(2^{2k},1),(2^{2k}+1,1)]\\
								&+ [(2^{2k},0),(2^{2k},1),(2^{2k}+1,1)]
								- [(2^{2k},0),(2^{2k}+1,0),(2^{2k}+1,1)],
\end{align*}
where the right-hand side of the first line is equal to $0$ since the 
triple $((1,0),(2^{2k}+1,0),(2^{2k}+1,1))$ is obtained from $((0,0),(2^{2k},0),(2^{2k},1))$
via the translation by $(1,0)$ (see Remark \ref{rmk: translation invariant} with $(i,j)=(1,0$)), the second line is equal to $f_*^{2^{2k}} (c)$, and the third line is equal to $-c$
(see again Remark \ref{rmk: translation invariant} with $(i,j)=(2^k,0)$).
\end{step}

\begin{figure}[h]
	\includegraphics{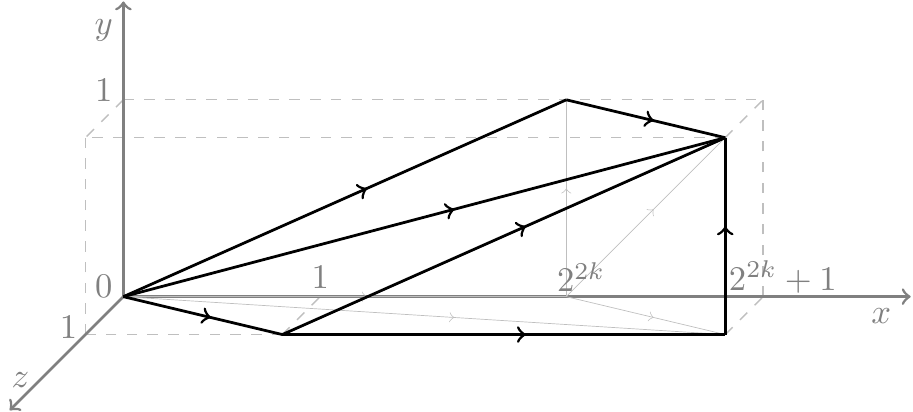}
	\caption{
		The chain $\tau_k$ is obtained by projecting the triangulation of the prism in the picture on the $xy$-plane.
	}\label{prysm}
\end{figure}

Let $S^1=\mathbb{R}/\mathbb{Z}$ and 
denote by 
$\gamma \colon S^1 = S^1 \times \{0 \} \hookrightarrow T = S^1 \times S^1$  the simple closed curve in $T$ defined by
$\gamma ([t])=[(t,0)]$.
The boundary $b_k$ is supported in the image of $\gamma$,  hence it represents a chain $s_k^0 \in C_2(S^1,\Z)$ in $S^1$ as well, where we identify $S^1$ with the image of $\gamma$.
Notice that with this notation we have $\gamma_* (s_k^0) = b_k$.

In order to find an efficient chain $\omega \in C_3(T,\Z)$ whose boundary is $b_k$, it is enough to find an efficient chain $\omega_{S^1} \in C_3(S^1,\Z)$ whose boundary is $s_k^0$ and set $\omega = \gamma_*(\omega_{S^1})$.

As we did for the torus, we denote by $\pi_{S^1}\colon \R \to \R/\Z = S^1$ the quotient map, and,  if $x_1,\ldots,x_{n +1} \in \R$, then we denote by $[x_1,\ldots ,x_{n +1}]$ the straight $n$-simplex in $S^1$ given by $\pi_{S^1 } (\str_1[x_1,\ldots,x_{n +1}])$.
According to this notation, we have
\[s_k^0 = \left[0,1,2^{2k}+1\right]-\left[0,2^{2k},2^{2k}+1\right]\ .\]
\begin{rmk}
	As in Remark \ref{rmk: translation invariant}, for every $i \in \Z$, the symbols $[x_1,\ldots ,x_{n +1}]$ and $[x_1 + i,\ldots ,x_{n +1} +i]$ denote the same simplex in $S^1$.
\end{rmk}
\begin{step}\label{step: passo induttivo}
	For every $i=0,\dots, 2k$ let
	\[s^i_k= \left[0,2^i,2^{{2k}- i}+2^i\right]- \left[0,2^{{2k}- i},2^{{2k}- i}+2^i\right]\ .\]
	Then there exist two chains $\alpha,\beta,\ \in C_3(T,\Z)$ (not depending on $k$ and $i$) and maps $\phi^i_{k} \colon  T \to S^1$ such that $  (\phi_{k}^i)_*(\partial\alpha)+(\phi_{k}^{i +1})_*(\partial\beta)= s_k^{i +1}-s_k^i$.
\end{step}
\begin{proof}
	Recall that
	\[c = [(0,0),(1,0),(1,1)] - [(0,0),(0,1),(1,1)]\]
	and set
	\begin{align*}
		a  = & + [(0,0),(1,0),(1,1/2)] - [(0,0),(0,1/2),(1,1/2)] +  \\
		     & + [(0,1/2),(1,1/2),(1,1)] - [(0,1/2),(0,1),(1,1)]\,, \\
		b =  & + [(0,0),(1/2,0),(1/2,1)]-[(0,0),(0,1),(1/2,1)]  +   \\
		     & + [(1/2,0),(1,0),(1,1)] - [(1/2,0),(1/2,1),(1,1)]\,.
	\end{align*}
	The chains $a,b,c$ are all fundamental cycles of the torus, thus they are cobordant (Figure \ref{fig: somecycles}).
	In particular, there exist $\alpha,\beta\in C_3(T,\Z)$ such that
	\begin{align*}\partial \alpha =& a - c \\ \partial \beta =& c - b \,.\end{align*}
	Let us set
	\[
		\begin{matrix}
			\widetilde{\phi}_k^i \colon & \R^2  & \to     & \R^1                     \\
			                            & (x,y) & \mapsto & 2^{i} x +2^{2k - i} y\ .
		\end{matrix}
	\]
	We define $\phi_k^i \colon  T =\R^2/\Z^2 \to S^1 =\R/\Z$ as the quotient map induced by $\widetilde{\phi}_k^i $. Then (see Figure \ref{fig: somecycles})
	\[\left(\phi_k^i \right)_* (c)=  \left[0,2^i,2^{2k - i}+2^i\right]- \left[0,2^{2k -i},2^{2k - i}+2^i\right]= s^i_k                                                   \]
	and
	\begin{align*}
		\left(\phi_k^i \right)_*(a)= & + \left[0,2^{i},2^{2k - i -1}+2^{i}\right] - \left[0,2^{2k - i -1},2^{2k - i -1}+2^{i}\right] +                              \\
		                             & + \left[2^{2k - i -1},2^{2k - i -1}+2^{i},2^{2k - i}+2^{i}\right]   - \left[2^{2k - i -1},2^{2k - i},2^{2k - i}+2^{i}\right] \\
		=                            & + \left[0,2^{i},2^{2k - i -1}+2^{i}\right] - \left[0,2^{2k - i -1},2^{2k - i -1}+2^{i}\right] +                              \\
		                             & + \left[2^{i},2^{i +1},2^{2k - i -1}+2^{i +1}\right]   - \left[2^{i},2^{2k - i -1}+2^i,2^{2k - i -1}+2^{i +1}\right]         \\
		=                            & \left(\phi_k^{i +1}\right)_* (b).
	\end{align*}
	Putting everything together, we have
	\begin{multline*}
		\left(\phi_k^i\right)_*\left(\partial \alpha\right)+ \left(\phi_k^{i +1}\right)_* \left(\partial \beta\right)=\left(\phi_k^i\right)_*\left(a - c\right)+ \left(\phi_k^{i +1}\right)_* \left(c - b \right)                                                                                     \\
		=                                                                  (\phi_k^{i +1})_*(b) - s_k^i + s_k^{i +1} -(\phi_k^{i +1})_* (b)= s_k^{i +1}-s_k^i\ .
	\end{multline*}
\end{proof}
\begin{figure}[h]
	\includegraphics{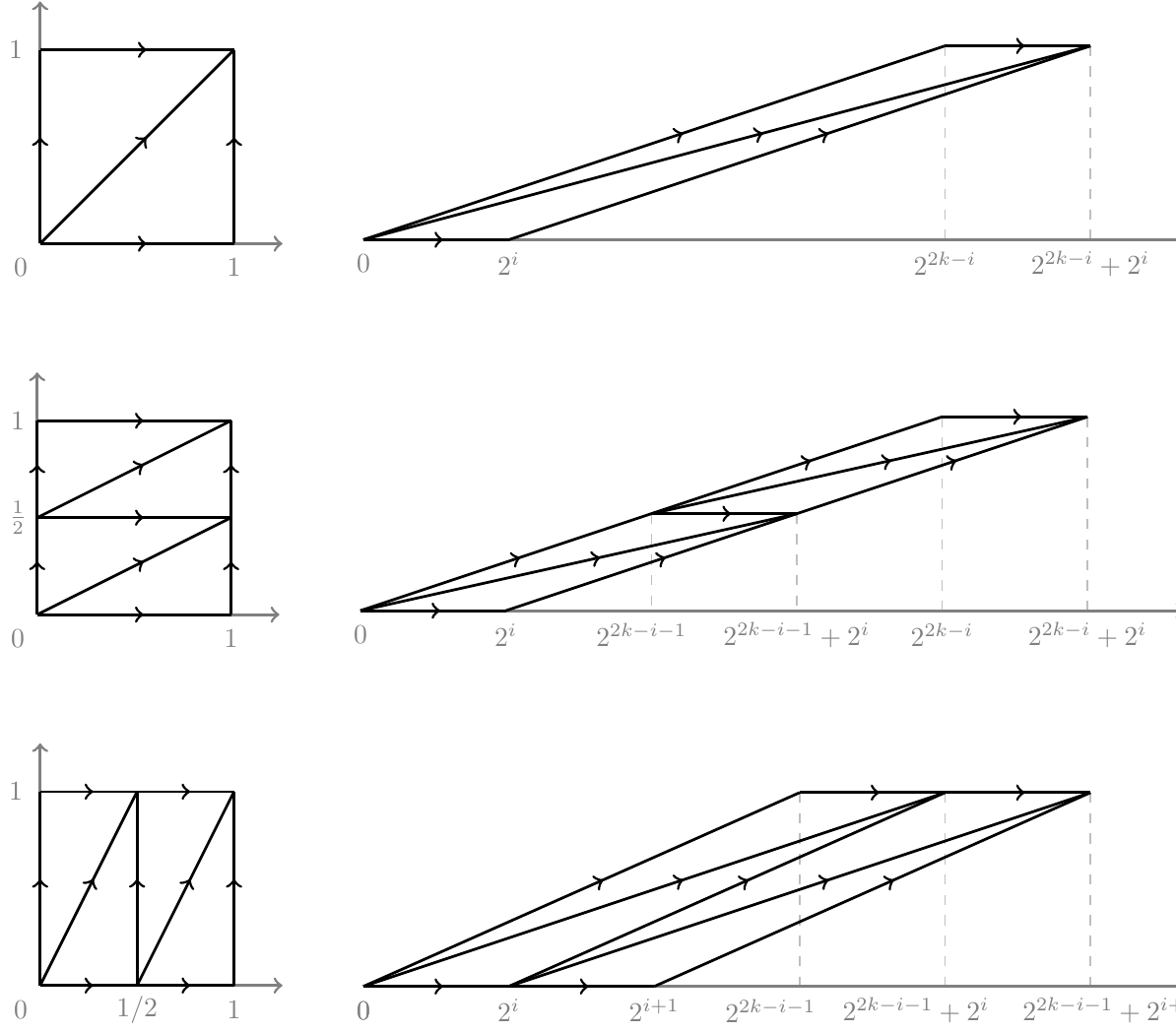}
	\caption{
		Lifts in $\R^2$ of the fundamental cycles $c,a,b$ are represented on the left. The parallelograms on the right, once projected on the horizontal axis, represent  lifts in $\R$ of the cycles $\phi_k^i(c),\phi_k^i(a),\phi_k^{i +1} (b)$. Notice that the lift of $\phi_k^{i +1} (b)$ is obtained from the lift of $\phi_k^i(a)$ by translating the upper parallelogram on the right of the lower parallelogram.
	}\label{fig: somecycles}
\end{figure}

\begin{step}\label{step: filling}
	Set
	\[\omega_k =\sum_{i = 0}^{k -1} (\phi_k^i)_*(\alpha)+ (\phi_k^i)_*(\beta)\ .\]
	Then
	\[\partial \omega_k =\sum_{i = 0}^{k -1} \left((\phi_k^i)_*(\partial\alpha)+ (\phi_k^{i +1})_*(\partial\beta)\right) = \sum_{i = 0}^{k -1} (s_k^{i +1}- s_k^i)= s_k^k -s_k^0 =- s_k^0 \,,	\]
	where in the last equality we used the fact that  $s_k^k = 0$.
	Moreover,
	\[
		\norml{w_k}\leq
		\sum_{i = 0}^{k -1}\Big( \norml{(\phi_k^i)_*(\alpha)}+ \norml{(\phi_k^i)_*(\beta)}\Big)\leq
		k \left(\norml{\alpha}+ \norml{\beta}\right)\ .
	\]
\end{step}
We are now ready to conclude the proof of Theorem \ref{thm:fv_Z-Dehn}.
Thanks to Step \ref{step: collasso} and Step \ref{step: filling} we have
\[
	f_*^{2^{2k}} (c)- c =
	\partial \tau_k + b_k =
	\partial \tau_k + \gamma_*(s^0_k) =
	\partial \tau_k + \gamma_*\left(-\partial \omega_k \right)=
	\partial \left( \tau_k - \gamma_*\left( \omega_k \right)\right)
\]
and
\[\norml{\tau_k - \gamma_*\left( \omega_k \right)}\leq
	\norml{\tau_k}+ \norml{\gamma_*\left( \omega_k \right)}\leq 3 + k \left(\norml{\alpha}+ \norml{\beta}\right)\ .\]
Thus
\begin{align*}
	\fv_\Z(f)= &
	\lim_{k \to \infty}\frac{\fnorm{f_*^{k} (c)- c }} {k}                                                   \\
	=          & \lim_{k \to \infty}\frac{\fnorm{f_*^{2^{2k}} (c)- c }} {2^{2k}}                            \\
	\leq       & \lim_{k \to \infty}\frac{\norml{\tau_k - \gamma_*\left( \omega_k \right)}} {2^{2k}}        \\
	\leq       & \lim_{k \to \infty}\frac{3 + k \left(\norml{\alpha}+ \norml{\beta}\right)} {2^{2k}} = 0\ . \\
\end{align*}
\qed

\subsection{Filling Volume of Dehn twists on higher genus  surfaces}
In this subsection we exploit the vanishing of the integral filling volume for Dehn twists on the torus to obtain a similar result for higher genus surfaces.

Let us first  fix some notation.
Let $S$ be a surface. With a slight abuse, if $\gamma\colon S^1 \to S$ is a closed curve, then
we denote by $\gamma$ also the image $\im(\gamma)\subset S$ of the map $\gamma$.
Moreover,
if $S'\subset S$ is a closed subset of $S$, then we denote by $S \cminus S'$ the completion of $S \setminus S'$ with respect to some auxiliary Riemannian metric on $S$.
In this way, when $\gamma$ is a simple closed curve, the surface $S \cminus \gamma$ is a (not necessarily connected) compact surface with two boundary components, both canonically identified with $\gamma$.

\begin{thm}\label{thm: DehnTwistSuperficie}
	Let S be a closed surface of genus $g \geq 1$, $\gamma \colon S^1 \to S$ be a simple closed curve in $S$ and $f \colon S \to S$ be the Dehn twist along $\gamma$. Then $\fv_\Z(f)= 0$.
\end{thm}

\begin{proof}
	By Theorem~\ref{thm: IntegralFillingVolumeAndCovering}, it is sufficient to show that
	\[ \lim_{k \to \infty} \frac{\|S\rtimes_{f^k} S^1\|_\Z} {k}=0 .\]
	To this aim, we first describe $S\rtimes_{f^k} S^1$ in an alternative way.

	The compact manifold $N'= ({S \cminus \gamma}) \times S^1$ has two boundary components $T_1, T_2$, both canonically identified
	with the torus $T=\gamma \times S^1$; we denote by
	\[i \colon  T_1 \hookrightarrow T_2 \]
	the map given by the composition of the two canonical identifications $T_1 \to T$ and $T \to T_2$. If we endow $T_1$ and $T_2$ with the orientations induced by $N'$, then the map $i$ is orientation-reversing.

	Let $c' \in C_3(N',\partial N', \Z)$ be a relative integral fundamental cycle of the pair $(N',\partial N')$.
	We have $\partial c' = b_1 + b_2$, where $b_i \in C_2(T_i,\Z)$ is a fundamental cycle.
	The chains $-i_*(b_1)$ and $b_2$ are both fundamental cycles of $T_2$, thus there exists $w \in C_3(T_2,\Z)\subseteq C_3(N',\Z)$ such that $\partial w =i_*(b_1)+b_2$.
	Set $c = c'-w$, so that $c$ is still a relative fundamental cycle of $(N',\partial N')$ and $\partial c = b_1 - i_*(b_1)$.

	Let now $g' \colon  T_1 \to T_1$ be the Dehn twist along $\gamma$, let $g_k \colon  T_1 \to T_2$ be the composition $i \circ (g')^k$, and $N_k$ be the manifold obtained from  $N'$ by gluing the two boundary components along $g_k$.
	Let us denote by $p \colon  N'\to N_k$ the projection.
	Then, $p_*(c)\in C_3(N_k,\Z)$ is a chain whose boundary in $N_k$ is contained inside $p(T_1)$, and it is given by
	\[
		\partial p_*(c)
		= p_*(b_1)- p_*(i_*(b_1))
		= p_*(b_1) - p_*(g'_*)^{- k}(b_1)
		= p_*(b_1 - (g'_*)^{- k}(b_1))\ .
	\]
In particular, the filling norm of $\partial p_*(c)$ as a boundary in $N_k$ is bounded from above by the filling norm of $p_*(b_1 - (g'_*)^{- k}(b_1))$ as a  boundary in $p(T_1)\cong T$.
	By \cite[Lemma 3.3]{foozwell_four-dimensional_2016}, the manifold $N_k$ is homeomorphic to the manifold $S\rtimes_{f^k} S^1$.
	Thus,
	\begin{align*}
		\fv_\Z(f)= & \lim_{k \to \infty} \frac{\|S\rtimes_{f^k} S^1\|_\Z} {k}                                               \\
		=          & \lim_{k \to \infty} \frac{\|N_k\|_\Z} {k}                                                              \\
		\leq       & \lim_{k \to \infty} \frac{\norml{p_*(c)}+ \fnorm{\partial p_*(c)}} {k}                                 \\
		\leq       & \lim_{k \to \infty} \frac{\norml{c}+ \fnorm{p_*(b_1 - (g'_*)^{- k}(b_1))}} {k}                         \\
		\leq       & \lim_{k \to \infty} \frac{\norml{c}} {k}+ \lim_{k \to \infty}\frac{\fnorm{b_1 - (g'_*)^{- k}(b_1)}}{k} \\
		=          & 0 + \fv_\Z(g')                                                                                         \\
		=          & 0 \,,
	\end{align*}
	where the last equality is due to Theorem \ref{thm:fv_Z-Dehn}.
\end{proof}

\subsection{Integral filling volume of surface automorphisms}
In this subsection we use the previous results to completely understand, in dimension 2, when the integral filling volume of an orientation preserving self-homeomorphism vanishes.

\begin{varthm}[Theorem~\ref{filvol:surfaces}]
	Let $S$ be the closed oriented surface of genus $g$ and let
	$f\colon S\to S$ be an orientation-preserving  self-homeomorphism of $S$. Then
	$\fvz(f)>0$ if and only if one of the following conditions is satisfied:
	\begin{itemize}
		\item  $g = 1$ and $f$ is Anosov;
		\item $g>1$ and there exist a natural number $k\in\mathbb{N}$ and a subsurface $S'\subseteq S$ (that is, a codimension-0  submanifold $S'$ of $S$, possibly with boundary) such that $f^k$ is isotopic to a homeomorphism
			$h\colon S \to S$ that satisfies $h(S')=S'$ and restricts to a pseudo-Anosov self-homeomorphism $h|_{S'}$ of $S'$.
	\end{itemize}

\end{varthm}

\begin{proof}
	We consider separately the case in which $S$ is a sphere, a torus or a surface of genus $g>1$.

	We denote by $\phi$ the class of $f$ in the mapping class group $\mcg(S)$ of $S$, and we recall from~\cite{bertolotti_length_2022}
	that the filling volume of $f$ only depends on its homotopy class $\phi$ (hence we can denote $\fv_\Z(f)$ by $\fv_\Z(\phi)$).

	If $S$ is a sphere, then $\mcg(S)$ is the trivial group and so the integral filling volume is always zero.

	If $S$ is a torus, then $\phi$ is periodic, Anosov or reducible \cite[Theorem 13.1]{farb_primer_2011}:
	\begin{itemize}
		\item if $\phi$ is periodic, then $\fv_\Z(\phi)= 0$ \cite[Theorem 2]{bertolotti_length_2022};
		\item if $\phi$ is Anosov, then $\fvz(\phi)>0$ \cite[Proposition 4.3]{bertolotti_length_2022};
		\item if $\phi$ is reducible, then it is represented by some power $h^k$ of some Dehn twist $h \colon  T \to T$ (\cite[Section 13.1]{farb_primer_2011}).
		      By Theorem  \ref{thm:fv_Z-Dehn} and \cite[Theorem 2]{bertolotti_length_2022}, it follows that $\fvz(\phi)= k \cdot\fvz(h)= 0$.
	\end{itemize}

	If $S$ is a surface of genus $g >1$, then $\phi$ is periodic, pseudo-Anosov or reducible \cite[Theorem 13.2]{farb_primer_2011}:
	\begin{itemize}
		\item If $\phi$ is periodic, then $\fv_\Z(\phi)= 0$ \cite[Theorem 2]{bertolotti_length_2022};
		\item if $\phi$ is pseudo-Anosov, then $\fvz(\phi)\geq \fv_\R(\phi)>0$ \cite[Corollary 9] {bertolotti_length_2022} (in this case we take $S'= S$ and $k = 1 $).
		\item If $\phi$ is reducible, then there exists some positive power $\phi^k$ of $\phi$ that can be represented by a map $h$ with the following property \cite[Corollary 13.3]{farb_primer_2011}: there exist pairwise disjoint annuli  $A_1,\ldots ,A_m \subset S$  such that $\restr{h} {A_i}$ is a power of the Dehn twist on $A_i$ (for every $1 \leq i \leq m$) and the restriction of $h$ on each connected component of $S \setminus \cup_{i = 1}^{m} A_i$ is either pseudo-Anosov or the identity.

		      If on at least one of these connected components, say $S'\subset S$, the restriction $\restr{h} {S'}$ of $h$ is pseudo-Anosov, then, by \cite[Corollary 9] {bertolotti_length_2022} and \cite[Theorem 2]{bertolotti_length_2022}, we have

		      \[\fvz(\phi)= \frac{\fvz(h)}{k}\geq \frac{ \fv_\R(h)}{k}>0 \,.\]

		      Otherwise, $h$ restricts to the identity on each connected component of $S \setminus \cup_{i = 1}^{m} A_i$, and if we denote by $g_i$ the Dehn twist on $A_i$, then

		      \[h = g_1^{k_1}\circ\ldots \circ g_m^{k_m}\,.\]
		      As $g_1,\ldots, g_m$ are Dehn twists on pairwise disjoint annuli, they pairwise commute.
		      It follows from \cite[Theorem 2]{bertolotti_length_2022} and Theorem \ref{thm: DehnTwistSuperficie} that
		      \[
			      \fvz(\phi) = \frac{\fvz(h)}{k}
			      \leq \frac{\sum_{i = 1}^m \left(k_i \cdot \fvz(g_i)\right)}{k}
			      = 0. \]
	\end{itemize}

\end{proof}

\section{Triangulation complexity of nilmanifolds and solmanifolds}\label{complexity:sec}
Recall from the introduction that all the manifolds appearing in this paper are assumed to be closed, oriented and connected.
Let $M$ be an $n$-manifold. Borrowing the terminology from low-dimensional topology, we
understand that a  \emph{triangulation}  of $M$ is an expression of $M$ as a union of $n$-dimensional simplices with some of their facets identified in pairs via affine homeomorphisms: in particular, self-adjacences of faces of the same simplex are allowed, i.~e.~closed simplices need not be embedded in $M$.
The complexity $\Delta(\mathcal T)$ of a triangulation $\mathcal T$ of $M$ is the number of $n$-simplices  in $\mathcal T$.

\begin{defn}\label{complexity:def}
	The \emph{$\Delta$-complexity} $\Delta(M)$ of a manifold $M$ is defined as the minimum of $\Delta(\mathcal{T})$, as $\mathcal{T}$ ranges
	over the  triangulations of $M$. If $M$ does not admit any triangulation, then we set $\Delta (M)=\infty$.
\end{defn}

\begin{rmk}
	It is easy to check that the second barycentric subdivision of a triangulation (in our sense) of $M$ yields a strict triangulation of $M$, i.~e.~a description of $M$
	as the geometric realization of a simplicial complex. As a consequence, if $\Delta'(M)$ denotes the minimal number of top-dimensional simplices in a strict triangulation of $M$,
	then
	$$
		\Delta(M)\leq \Delta'(M)\leq (n+1!)^2 \Delta(M)\ ,
	$$
	where $n=\dim M$.
	In particular, the numerical invariants $\Delta'$ and $\Delta$ are equivalent, according to the definition given in the introduction.
\end{rmk}

This section is devoted to the proof of Theorem~\ref{thm: triangulation complexity of nilmanifolds}, which ensures that, if
$\phi \colon T \to T$ is a non-trivial power of a Dehn twist or an Anosov homeomorphism of the torus, then
the $\Delta$-complexity
of the torus bundle $T\rtimes_{\varphi^n} S^1$ grows linearly with respect to $n$.
As mentioned in the introduction, the Anosov case is already treated in \cite{lackenby_triangulation_2022}, where a quantitative estimate on the growth is also provided. Hence, our contribution mainly consists in adapting Lackenby's and Purcell's strategy to the case of Dehn twists.

Before going into the proof of Theorem~\ref{thm: triangulation complexity of nilmanifolds} we introduce
the spine graph of a surface.
As discussed in~\cite{lackenby_triangulation_2019,lackenby_triangulation_2022}, this object
plays a fundamental role in the computation of the $\Delta$-complexity of surface bundles over the circle.

\subsection{Spine graph on a surface}
Let $S$ be a closed surface.
A \emph{spine} $\Gamma$ on $S$ is an embedded graph without vertices of degree $1$ and $2$, and whose complement $S \cminus \Gamma$ is a disc. If $S$ is endowed with a fixed cellular structure, then we say that the spine $\Gamma$ is cellular if it is a subcomplex of $S$.

An \emph{edge contraction} on a spine $\Gamma$ of $S$ is the move inside $S$ that collapses an edge of $\Gamma$ to a point (thus collapsing two vertices of $\Gamma$ to a unique vertex).
An \emph{edge expansion} is the reverse of this operation.

The \emph{spine graph} $\spl(S)$  on a closed surface $S$ is the graph whose vertices are spines on $S$, considered up to isotopy, in which two vertices are joined by an edge if and only if they differ by an edge expansion/contraction.

Note that the mapping class group $\mcg(S)$ acts by isometries on 
$\spl(S)$.
\begin{prop}[{\cite[Proposition 2.7]{lackenby_triangulation_2019}}]\label{prop: quasi-isometry 3 graphs}
	Let $S$ be a closed surface of genus $g \geq 1$, and  fix  a
	spine $\Gamma$  in $S$. Then, the map
	\[
		\begin{matrix}
			\mcg(S) & \to     & \spl(S)             \\
			\gamma  & \mapsto & \gamma \cdot \Gamma
		\end{matrix}
	\]
	is a quasi-isometry between  the mapping class group $\mcg(S)$ and the spine graph $\spl(S)$.
\end{prop}

The statement of {\cite[Proposition 2.7]{lackenby_triangulation_2019}} only ensures that $\mcg(S)$ 
 and $\spl(S)$ are quasi-isometric to each other.
However, in the proof Lackenby and Purcell show that the action of $\mcg(S)$ on $\spl(S)$ verifies all the hypotheses of the Milnor-\v{S}varc Lemma and thus the latter guarantees that the map described in Proposition~\ref{prop: quasi-isometry 3 graphs} is a quasi-isometry.

\begin{lemma}\label{lemma: linear growth Dehn twist on spines}
	Let $S$ be a closed surface and let $\phi \in \mcg(S)$ be an element of infinite order.
	Then there exists a constant $k_\spin>0$ such that, for every spine $\Gamma$ of $S$ and every $n\in\mathbb{N}$, we have
	$$
		\frac{d_{\spl(S)}(\phi^n (\Gamma), \Gamma)}{n}\geq k_\spin\ .
	$$
\end{lemma}
\begin{proof}
	Let us fix a spine $\Gamma$ of $S$. The map
	$$
		n\mapsto d_{\spl(S)}(\phi^n (\Gamma), \Gamma)\ ,
	$$
	is subadditive, hence, by Fekete's Lemma, the limit $$k_\spin=\lim_{n \to \infty} \frac{d_{\spl(S)}(\phi^n(\Gamma),\Gamma)} {n}$$
	exists and is equal to $\inf\{d_{\spl(S)}(\phi^n(\Gamma),\Gamma)/n\, ,\ n\geq 1\}$.

	Let us first show that $k_\spin$ does not depend on the chosen spine $\Gamma$. Indeed, if $\Gamma'$ is any spine of $S$, then
	\begin{align*}
		d_{\spl(S)} (\phi^n(\Gamma),\Gamma)\leq &
		d_{\spl(S)} (\phi^n(\Gamma),\phi^n(\Gamma')) + d_{\spl(S)} (\phi^n(\Gamma'),\Gamma')+d_{\spl(S)} (\Gamma',\Gamma) \\
		=                                       & d_{\spl(S)} (\phi^n(\Gamma'),\Gamma')+2d_{\spl(S)} (\Gamma',\Gamma)
	\end{align*}
	and, analogously,
	\[
		d_{\spl(S)} (\phi^n(\Gamma'),\Gamma')\leq
		d_{\spl(S)} (\phi^n(\Gamma),\Gamma)+2d_{\spl(S)} (\Gamma',\Gamma)\ .
	\]
	These inequalities readily imply that
	\[
		\liminf_{n \to \infty} \frac{d_{\spl(S)} (\phi^n(\Gamma),\Gamma)} {n} =
		\liminf_{n \to \infty}\frac{d_{\spl(S)} (\phi^n(\Gamma'),\Gamma')} {n}\ ,
	\]
	i.~e.~$k_\spin$ does not depend on $\Gamma$.

	In order to conclude  we are left to show that $k_\spin>0$.
	To this aim, let us fix a finite set of generators of $\mcg(S)$.
	By \cite[Theorem 1.2]{farb_rank-1_2001}, elements of infinite order in $\mcg(S)$ are undistorted, i.~e.~there exists a constant $c>0$ such that
	\[d_{\mcg(S)} (\phi^n,\id)\geq c n \,.\]

	By Proposition \ref{prop: quasi-isometry 3 graphs}, there are two constants $a >0,b \geq 0$ such that
	\[d_{\spl(S)} (\phi^n(\Gamma),\Gamma)\geq a \cdot d_{\mcg(S)} (\phi^n,\id)- b \geq acn - b\ , \]
	hence
	$$
		k_\spin=\lim_{n\to \infty}\frac{d_{\spl(S)}(\phi^n(\Gamma),\Gamma)} {n}\geq \lim_{n\to \infty}\frac{acn-b} {n}=ac>0\ .
	$$
	\end{proof}

\subsection{Triangulation complexity of nilmanifolds and solmanifolds}
In this subsection we prove Theorem \ref{thm: triangulation complexity of nilmanifolds}, which we recall here for the convenience of the reader:

\begin{varthm}[Theorem~\ref{thm: triangulation complexity of nilmanifolds}]
	Let $f \colon T \to T$ be a self-homeomorphism of the torus representing an infinite-order element in $\mcg(T)$. Then  the $\Delta$-complexity  of the mapping torus of $f^i$ has linear growth with respect to $i$, i.~e.~there exist constants $0<k_1\leq k_2$ such that
	$$
		k_1\leq \frac{\Delta \left(T \rtimes_{f^i} S^1\right)}{i}\leq k_2
	$$
	for every $i\geq 1$.
\end{varthm}

Let $f \colon T \to T$ be a homeomorphism of the torus, and let $\mathcal T$ be a triangulation of
$$
	M = T\rtimes_f S^1= T \times [0,1]/\sim\ ,
$$
where $(x,0)\sim(f(x),1)$.

We denote by $\mathcal{H}'$ the dual handle decomposition induced by $\mathcal{T}$,
whose $i$-dimensional handles, $i=0,1,2,3$, are obtained from a thin neighbourhood of the $(3-i)$-skeleton of $\mathcal{T}$ by removing a thin
neighbourhood of the $(3-i-1)$-skeleton (see e.g.~\cite[Section 2]{lackenby_triangulation_2022}).

Recall that the \emph{weight} of a surface $S$  embedded in $M$
and in general position with respect to $\mathcal T$ is defined as the cardinality of the intersection between $S$ and the $1$-skeleton of $\mathcal T$.
We choose a fiber $S\subseteq M$ of least weight with respect to $\mathcal{T}$. Up to isotopy, we may assume that $S$ is in normal position with respect to the handle decomposition $\mathcal  H'$ (see~\cite[Section 4]{lackenby_triangulation_2019} for the definition of normal position).

As explained in~\cite[Section 4]{lackenby_triangulation_2019}, the surface
$S$ intersects the handles of $\mathcal H'$ in properly embedded discs, thus
inheriting a cellular structure form $\mathcal H'$. Moreover, the handle decomposition $\mathcal{H}'$ of $M$ induces a handle decomposition $\mathcal{H}$
of the compact manifold with boundary $M \cminus S$ obtained by cutting $M$ along $S$.
Since $S$ is a fiber of a torus bundle over the circle, we may identify $M\cminus S$ with
the product $T\times [0,1]$.
In this way we obtain natural identifications between $S$ and $T\times \{0\}$, $T\times \{1\}$. If we identify both $T\times \{0\}$ and $T\times \{1\}$ with the torus $T$ via the obvious projections,
then the cellular structure of $S$ induces two cellular structures on $T$, which are taken one onto the other by the homeomorphism $f$. Henceforth, we fix on $T$ the cellular
structure induced by $S$ via the identification $S\cong T\times \{0\}$.

We choose a   cellular spine $\Gamma_0$ of $T$. By construction, the spine $f(\Gamma_0)$ is cellular with respect to the
cellular structure induced by the identification of $T$ with $T\times \{1\}$. Therefore, by
\cite[Lemma 2.2]{lackenby_triangulation_2022} (which ensures that the handle decomposition $\mathcal{H}$ is \emph{pre-tetrahedral} according to the terminology
therein),
\cite[Theorem 2.16]{lackenby_triangulation_2022} (which implies that $\mathcal{H}$ admits no annular simplification, see again~\cite{lackenby_triangulation_2022} for the definition of such a notion) and \cite[Lemma 2.5 and Theorem 9.1]{lackenby_triangulation_2022} (applied to $f(\Gamma_0)$, with the roles of $T\times \{0\}$ and $T\times \{1\}$ reversed), there exists a cellular spine $\Gamma_1$ such that
\[d_{\spl(T)} (f(\Gamma_0),\Gamma_1)\leq k_{\hand} \Delta(\mathcal T),\]
where $k_{\hand}$ is a universal constant, which is indipendent of the choosen cellular spine $\Gamma_0$ (see \cite[Theorem 9.1]{lackenby_triangulation_2022}).
By iterating this process, we obtain a sequence of cellular spines $\{\Gamma_i \}$ satisfying
\[d_{\spl(T)} (f(\Gamma_{i -1}),\Gamma_i)\leq k_{\hand} \Delta(\mathcal T),\]
where at each step of the iteration we take as input $\Gamma_{i -1}$ instead of $\Gamma_0$.
As there are only finitely many cellular spines in $T \times \{0 \}$, there are two integers $r<s$ such that $\Gamma_r =\Gamma_s$.
By relabeling, we can assume $r = 0,s = m$.
Thus,
\begin{align*}
	d_{\spl(T)}(f^m(\Gamma_0),\Gamma_0)= & d_{\spl(T)}(f^m(\Gamma_0),\Gamma_m)                                                \\
	\leq                                 & \sum_{i = 0}^{m -1}d_{\spl(T)}(f^{m - i}(\Gamma_{i}),f^{m - i -1} (\Gamma_{i +1})) \\
	=                                    & \sum_{i = 0}^{m -1}d_{\spl(T)}(f(\Gamma_{i}),\Gamma_{i +1})                        \\
	\leq                                 & m k_{\hand} \Delta(\mathcal T)\ .
\end{align*}

Now, let $f \colon T \to T$ be a self-homeomorphism representing an infinite-order element in $\mcg(T)$, and let $\mathcal T_i$ be a triangulation of $M_i=T\rtimes_{f^i} S^1$
such that $\Delta(M_i)=\Delta (\mathcal{T}_i)$.
By the inequality above, there exist a positive integer $m_i$ and a cellular spine $\Gamma_0^{(i)}$ such that
\begin{equation}\label{eq: triang comlexity power dehn twist}
	\Delta(M_i)=\Delta(\mathcal T_i)\geq
	\frac{d_{\spl(T)} \left((f^{i})^{m_i}\left(\Gamma_0^{(i)}\right),\Gamma_0^{(i)}\right)} {m_ik_{\hand}}\geq \frac{ k_\spin}{k_\hand}i\ ,
\end{equation}
where the last inequality is due to Lemma \ref{lemma: linear growth Dehn twist on spines}.
This proves  that $\Delta(M_i)$ grows at least linearly in $i$.

On the other hand, if $\mathcal T$ is a triangulation of $M$, then we can lift it to the $i$-sheeted cover $M_i \to M$, obtaining a triangulation of $M_i$ with $i \Delta(\mathcal T)$ tetrahedra. Therefore, $\Delta(M_i)\leq i \Delta(\mathcal T)$, proving that the function $\Delta(M_i)$ grows at most linearly with respect to  $i$.

\bibliography{math_papers}
\bibliographystyle{alpha}
\end{document}